\documentclass{amsart}

\usepackage{graphicx}

\newtheorem{theorem}{Theorem}[section]
\newtheorem{lemma}[theorem]{Lemma}
\newtheorem{corollary}[theorem]{Corollary}
\theoremstyle{definition}
\newtheorem{definition}[theorem]{Definition}

\newtheorem{remark}[theorem]{Remark}

\numberwithin{equation}{section}

\begin{document}

\title{Periodic sequences modulo $m$}

\author{Alexandre Laugier}
\address{Lyc{\'e}e professionnel Tristan Corbi{\`e}re, 16 rue de Kerv{\'e}guen - BP 17149, 29671 Morlaix cedex, France}
\email{laugier.alexandre@orange.fr}

\author{Manjil P. Saikia}
\address{Fakult\"at f\"ur Mathematik, Universit\"at Wien, Oskar-Morgenstern-Platz 1, 1090 Wien, Austria}
\email{manjil.saikia@univie.ac.at}


\date{\today}


\newcommand{\llbracket}{[[}
\newcommand{\rrbracket}{]]}

\begin{abstract}
We give a few remarks on the periodic sequence $a_n=\binom{n}{x}~(mod~m)$ where $x,m,n\in \mathbb{N}$, which is periodic with minimal length of the period being $$
\ell(m,x)={\displaystyle\prod^w_{i=1}p^{\lfloor\log_{p_i}x\rfloor+b_i}_i}
=m{\displaystyle\prod^w_{i=1}p^{\lfloor\log_{p_i}x\rfloor}_i}
$$ where  $m=\prod^w_{i=1}p^{b_i}_i$. We prove certain interesting properties of $\ell(m,x)$ and derive a few other results and congruences.
\end{abstract}

\maketitle

\section{{Introduction and Preliminaries}}

This paper deals with the periodicity of binomial coefficients which have previously been studied by many mathematicians, \cite{fibo} and \cite{zabek} are some examples of results obtained in this direction. The authors in \cite{mps} stated and proved the following theorem.

\begin{theorem}\label{mps-jv}
A natural number $p>1$ is a prime if and only if $\binom{n}{p}-\lfloor\frac{n}{p}\rfloor$ is divisible by $p$ for every non-negative $n$, where $n>p+1$ and the symbols have their usual meanings.
\end{theorem}

The proof of Theorem \ref{mps-jv} was completed by Laugier and Saikia \cite{almps}. In this section we state without proof the following results which we shall be referring in the coming sections. The proofs can be found in \cite{mps2}.

\begin{definition}
 A sequence $(a_n)$ is said to be periodic modulo $m$ with period $k$ if there exists an integer $N>0$ such that for all $n>N$ $$a_{n+k}=a_n \pmod m.$$
\end{definition}

\noindent In the following, we shall use usual periodicity with $N=1$ unless otherwise mentioned.

\begin{theorem}\label{pmod}
The sequence $(a_n)=\binom{n}{x}~(mod~m)$ is periodic, where $x,m,n \in \mathbb{N}$.
\end{theorem}

\begin{theorem}\label{lpmod}
For a natural number $m=\prod^k_{i=1}p_i^{b_i}$, the sequence $a_n\equiv \binom{n}{m}~(mod~m)$ has a period of minimal length,

$$l(m)=\prod^k_{i=1}p_i^{\lfloor \log_{p_i}m \rfloor+b_i}.$$
\end{theorem}

\noindent Theorem \ref{lpmod} was also derived in \cite{lu}, however the motivation of that paper was quite different from \cite{mps2}.

The following generalization of Theorem \ref{mps-jv} was also proved in \cite{almps}.

\begin{theorem}\label{genl}
For natural numbers $n, k$ and a prime $p$ we have the following $$\binom{n}{p^k}-\displaystyle\left\lfloor \frac{n}{p^k}\displaystyle\right\rfloor \equiv 0~(\textup{mod}~p).$$
\end{theorem}

We also fix the notation $[[1,i]]$ for the set $\{1, 2, \ldots, i\}$ throughout the paper.

\begin{definition}\label{ord}
We define $\textup{ord}_p(n)$ for $n\in \mathbb{N}$ to be the greatest exponent of $p$ with $p$ a
prime in the
decomposition of $n$ into prime factors, $$\textup{ord}_p(n)=\max\left\{k\in\mathbb{N}\,:\,p^k|n\right\}.$$
\end{definition}

\section{{Results and Discussion}}

\subsection{{Remarks on Theorem \ref{pmod}}}

The integer $n$ in Theorem \ref{pmod} should be greater than $x$. Otherwise, the binomial
coefficient $\binom{n}{x}$ is not defined. But, we can
extend the definition of $\binom{n}{x}$ to integer $n$ such that
$0\leq n<x$ by setting $\binom{n}{x}=0$ if $0\leq n<x$. Nevertheless,
notice that this extension is not necessary in order to prove this
theorem about periodic sequences.

The case where $m=0$ is not possible since the sequence
$(\binom{n}{x})$ is not periodic modulo $0$ or is not simply periodic.
So, if $x=m$, $x$ should be non-zero.

If $x=0$, then we have
$$
a_n\equiv a_{n+1}\equiv\ldots\equiv a_{n+k}\equiv 1\pmod m
$$
for any integers $n$ and $k$. So, if $x=0$, the sequence $(a_n)$ is
periodic with minimal period equal to $1$. We recall that if a
sequence is periodic, a period of such a sequence is a non-zero integer.

In the following, we assume $x\geq 1$.

\begin{lemma}\label{lm1}
For $n\geq x+1$
\begin{equation}
{\displaystyle\sum^{n-1}_{i=x}\binom{i}{x}}=\binom{n}{x+1}.\nonumber
\end{equation}

\end{lemma}

\noindent The proof of the above is not difficult and can be done using induction. We omit the details here.

Let $k$ be the length of a period of sequence
$a_n\equiv\binom{n}{x}\pmod m$, meaning
$\binom{n+k}{x}\equiv\binom{n}{x}\pmod m$. Then we have,

\begin{lemma}\label{lm2}
\begin{equation}
{\displaystyle\sum^{y+mk-1}_{j=y}\binom{j}{x}}\equiv 0\pmod m.\nonumber
\end{equation}

\end{lemma}

\begin{proof}
It is enough to notice the following

$$\sum_{j=y}^{y+mk-1}\binom{j}{x}=\sum_{i=0}^{m-1}\sum_{j=y}^{y+k-1}\binom{j+ik}{x}\equiv \sum_{i=0}^{m-1}r=mr \equiv 0~(\textup{mod}~m),$$

\noindent where $\sum_{j=y}^{y+k-1}\binom{j+ik}{x} \equiv r \pmod m$, for some $r$.
\end{proof}

In \cite{almps}, the authors mention without proof the following generalization of Theorem \ref{lpmod}.

\begin{theorem}\label{glpmod}
For a natural number $m=\prod^w_{i=1}p^{b_i}_i$, the sequence
$(a_n)$ such that $a_n\equiv\binom{n}{x}\pmod m$ has a period of
minimal length
$$
\ell(m,x)={\displaystyle\prod^w_{i=1}p^{\lfloor\log_{p_i}x\rfloor+b_i}_i}
=m{\displaystyle\prod^w_{i=1}p^{\lfloor\log_{p_i}x\rfloor}_i}.
$$
\end{theorem}

\noindent The proof follows from the proof of Theorem \ref{lpmod} as given in \cite{mps2} and also via Theorem 3 in \cite{lu}. An easy corollary mentioned in \cite{almps} is proved below.

\begin{corollary}
For $m=\prod^w_{i=1}p^{b_i}_i$ we have
$$
m^2\leq \ell(m)\leq m^{w+1}.
$$

\end{corollary}

\begin{proof}

We have $$\ell(m)=m\prod_{i=1}^w p_i^{\lfloor \log_{p_i}(m)\rfloor} \geq m\prod_{i=1}^w p_i^{\lfloor \log_{p_i}(p_i^{b_i})\rfloor}=m\prod_{i=1}^wp_i^{b_i}$$ and 
$$\ell(m)=m\prod_{i=1}^w p_i^{\lfloor \log_{p_i}(m)\rfloor} \leq m\prod_{i=1}^w p_i^{\log_{p_i}(m)}=m^{w+1}.$$
\end{proof}

\begin{remark}

Here $w\leq m-\varphi(m)$ where $\varphi$ is the Euler
totient function.

\end{remark}

We now formally give the following definition.

\begin{definition}[Minimal Period of a periodic sequence]\label{dfn}

The period of minimal length of a periodic sequence $(a_n)$ such that
$a_n\equiv\binom{n}{x}\pmod m$ with
$x\in\mathbb{N}$ and $m\in\mathbb{N}$, is the minimal
non-zero natural number $\ell(m,x)$ such that
for all positive integer $n$ we have
$$
\binom{n+\ell(m,x)}{x}\equiv\binom{n}{x}\pmod m
$$
where it is understood that
$$
\binom{n}{x}=\left\{\begin{array}{ccc}
0, & {\rm if} & 0\leq n<x\\
\frac{n!}{x!(n-x)!}, & {\rm if} & n\geq x.
\end{array}
\right.
$$

\end{definition}

\begin{remark}
If $x=0$, then $\ell(m,x=0)=1$ with $m\in\mathbb{N}$.
\end{remark}

From Definition \ref{dfn}
$$
\binom{\ell(m,x)}{x}\equiv\binom{\ell(m,x)+1}{x}\equiv\cdots\equiv
\binom{\ell(m,x)+x-1}{x}\equiv 0\pmod m.
$$
If $x>0$ ($x\in\mathbb{N}$), since any number
is divisible by $1$, we have
$$
\binom{x}{x}\equiv\binom{x+1}{x}\equiv\cdots\equiv\binom{2x-1}{x}\equiv
0\pmod 1.
$$
Regarding the definition of $\ell(m,x)$, since $x$ is the least
non-zero natural number which verifies this property, we can set
($x\in\mathbb{N}$)  $\ell(1,x)=1$.

The minimal period $\ell(m)$ of a sequence $(a_n)$ such that
$a_n\equiv\binom{n}{m}\pmod m$ with $m\in\mathbb{N}$ (see Theorem \ref{lpmod}) is given
by $\ell(m)=\ell(m,m)$.

Before we mention a few results we recall that $\log_ax=\frac{\ln x}{\ln a}$ and $\ln (1+x)={\displaystyle\sum_{k=1}^\infty} \frac{(-1)^{k+1}}{k}x^k$.

\begin{theorem}\label{lgt}
$$
\lfloor\log_p(x+1)\rfloor=\displaystyle\left\lfloor\log_p(x)+\frac{1}{x\ln p}\displaystyle\right\rfloor
=\left\{\begin{array}{ccc}
\lfloor\log_p(x)\rfloor, & {\rm if} & x\neq p^c-1;\\
\lfloor\log_p(x)\rfloor+1, & {\rm if} & x=p^c-1,
\end{array}
\right.
$$
with $c\in\mathbb{N}$.
\end{theorem}

\noindent The proof is not difficult and is an easy calculus exercise, so we shall omit it here.

We now have the following

\begin{corollary}\label{crl}

$$
\ell(m,x+1)=\left\{\begin{array}{ccccc}
\ell(m,x),  & {\rm if} & x\neq p^c-1 & {\rm and} & p|m;\\
p\,\ell(m,x), & {\rm if} & x=p^c-1 & {\rm and} & p|m,\\
\end{array}
\right.
$$
with $x,m\in\mathbb{N}$.
\end{corollary}

The proof of the above corollary comes from Definition \ref{dfn} and Theorem \ref{lgt}.

From Lemma \ref{lm1}
$$
{\displaystyle\sum^{x+k-1}_{j=x}\binom{j}{x}}=\binom{x+k}{x+1}.
$$
The binomial c{\oe}fficient
$\binom{x+k}{x+1}$ is well defined for $x\in\mathbb{N}$. Nevertheless,
it was remarked in \cite{mps2} that we can extend possibly the definition of $\binom{n}{x}$ (where it is implied
that $0\leq x\leq n$) to negative $n$.

Below we discuss a few general results and give a few general comments.

Using Pascal's rule, we can observe that
$$
\binom{x+k}{x+1}+\binom{x+k}{x}=\binom{x+k+1}{x+1}.
$$
Since $\binom{x+k}{x}\equiv \binom{x}{x}\equiv 1\pmod m$, we obtain
\begin{equation}\label{eqs}
\binom{x+k}{x+1}+1\equiv\binom{x+k+1}{x+1}\pmod m.
\end{equation}

If $x\neq p^c-1$ and $p|m$, then from the corollary above, we have
$k=\ell(m,x)=\ell(m,x+1)$. So
$$
\binom{x+k+1}{x+1}\equiv \binom{x+1}{x+1}\equiv 1\pmod m,
$$
and hence
$$
\binom{x+\ell(m,x)}{x+1}\equiv 0\pmod m.
$$

If $x=p^c-1$ and $p|m$, then from the corollary above, we have
$pk=p\,\ell(m,x)=\ell(m,x+1)$. Now from Theorem \ref{genl} we have
for $x=p^c-1$ and $m=p$ a prime with $c\in\mathbb{N}$,
\begin{equation}\label{eqp}
\binom{x+k+1}{x+1}
=\binom{p^c+\ell(p,p^c-1)}{p^c}\equiv\left\lfloor\frac{p^c+\ell(p,p^c-1)}{p^c}\right\rfloor
\equiv\left\lfloor\frac{\ell(p,p^c-1)}{p^c}\right\rfloor+1\pmod p.
\end{equation}
From \eqref{eqs} and \eqref{eqp}  with
$x=p^c-1$, $k=\ell(m,x)$ and $m=p$ a prime we have
$$
\binom{p^c-1+\ell(p,p^c-1)}{p^c}\equiv
\left\lfloor\frac{\ell(p,p^c-1)}{p^c}\right\rfloor\pmod p.
$$
We have  $\ell(p,p^c)=p^{c+1}=p\,\ell(p,p^c-1)$, so it follows that $\ell(p,p^c-1)=p^c$ and hence
$\lfloor\frac{\ell(p,p^c-1)}{p^c}\rfloor=1$ for
$c\in\mathbb{N}$. Thus
$$
\binom{2p^c-1}{p^c}\equiv 1\pmod p.
$$

Thus, we now have the following result.

\begin{theorem}
For a prime $p$ and a natural number $c$, we have 
 $$
\binom{2p^c-1}{p^c}\equiv 1\pmod p.
$$
\end{theorem}

In general, if $x=p^c-1$ and $p|m$, then since
$\lfloor\log_{p}(p^c)\rfloor=c$,
and from Corollary \ref{crl} we have
$$
\ell(m,p^c-1)=\frac{\ell(m,p^c)}{p}=mp^{c-1}{\displaystyle
\prod_{i\in\llbracket 1,k\rrbracket\,|\,p_i\neq p}
p^{\lfloor\log_{p_i}(p^c)\rfloor}_i}.
$$
If $b_i={\rm ord}_{p_i}(m)=\lfloor\log_{p_i}(p^c)\rfloor$ for $i\in\llbracket
1,k\rrbracket\,|\,p_i\neq p$ and
$b={\rm ord}_p(m)$, we have
$$
\ell(m,p^c-1)=\frac{\ell(m,p^c)}{p}=m\,p^{c-1}{\displaystyle
\prod_{i\in\llbracket 1,k\rrbracket\,|\,p_i\neq p}
p^{b_i}_i}=m\,p^{c-1}\times\frac{m}{p^b}.
$$
So, we deduce that
$$
\ell(m,p^c-1)=m^2\,p^{c-b-1}=\frac{m^2}{p^{b+1}}p^c,
$$
and
$$
\ell(m,p^c)=m^2\,p^{c-b}.
$$

In particular, when $m=p^b$ we have
$\ell(m=p^b,p^c-1)=p^{b+c-1}$. And, we get
$\ell(m=p,p^c-1)=p^c$. If $c\geq {\rm ord}_p(m)+1$, then $\ell(m,p^c-1)$ is
divisible by $m^2$. If $b=c$, then $\ell(m=p^c,p^c-1)=\frac{m^2}{p}$.

Now from \eqref{eqp} and the preceeding paragraph we have,
$$
\binom{x+k+1}{x+1}
=\binom{p^c+\ell(m,p^c-1)}{p^c}\equiv\left\lfloor
\frac{p^c+\ell(m,p^c-1)}{p^c}\right\rfloor
\equiv\left\lfloor\frac{m^2}{p^{b+1}}\right\rfloor+1\pmod p.
$$
From \eqref{eqs} with
$x=p^c-1$, $k=\ell(m,x)$, and also from the fact that $d\equiv e\pmod m$ and $p|m$ implies that
$d\equiv e\pmod p$ (the converse is not always true), we have
for $p|m$,
$$
\binom{p^c-1+m^2_c\,p^{c-b-1}}{p^c}\equiv
\left\lfloor\frac{m^2_c}{p^{b+1}}\right\rfloor\pmod p.
$$

\subsection{{Remarks on Theorem \ref{lpmod}}}

In the proof of Theorem \ref{lpmod}, the authors in \cite{mps2} first proved that a period of
a sequence $(a_n)$ such that $a_n\equiv\binom{n}{m}\pmod m$ with
$m=\prod^k_{i=1}p^{b_i}_i$, should be a
multiple of the number
$\ell(m)=m\prod^k_{i=1}p^{\lfloor\log_{p_i}(m)\rfloor}_i$. Afterwards, it is
proved that $\ell(m)$ represents really the minimal period of
such a sequence namely for every natural number $n$,
$$
\binom{n+\ell(m)}{m}\equiv\binom{n}{m}\pmod m.
$$
For that, the authors notice that it suffices to prove
$$
\frac{\prod^{m-1}_{i=0}(n-i)}{\prod^k_{j=1}p^{\vartheta_{p_j}(m)}_j}
\equiv\frac{\prod^{m-1}_{i=0}(n+\ell(m)-i)}{\prod^k_{j=1}p^{\vartheta_{p_j}(m)}_j}\pmod m,
$$
where $\vartheta_{p_j}(m)$ is the $p_j$-adic ordinal of $m!$ defined as
\begin{equation}\label{eqa}
\vartheta_{p_j}(m)={\rm ord}_{p_j}(m!)={\displaystyle\sum_{l\geq 1}
\left\lfloor\frac{m}{p^l_j}\right\rfloor}={\displaystyle
\sum^{\lfloor\log_{p_j}(m)\rfloor}_{l=1}
\left\lfloor\frac{m}{p^l_j}\right\rfloor}.
\end{equation}
Thus to prove Theorem \ref{lpmod} it is sufficient to show
$$
\frac{\prod^{m}_{i=1}(n-i+1)}{\prod^k_{j=1}p^{\vartheta_{p_j}(m)}_j}
\equiv\frac{\prod^m_{i=1}(n+\ell(m)-i+1)}
{\prod^k_{j=1}p^{\vartheta_{p_j}(m)}_j}\pmod m.
$$

Then, the authors observe that among the numbers $n,n-1,\ldots,n-m+1$,
there are at least $\lfloor\frac{m}{p^l}\rfloor$ that are divisible by
$p^l$ for every positive integer $l$ and any prime $p$ which appears
in the prime factorization of $m$. In particular, if $p$ divides $m$,
we can notice that among the numbers $n,n-1,\ldots,n-m+1$ (which
represents $m$ consecutive numbers),
there are exactly $\lfloor\frac{m}{p}\rfloor=\frac{m}{p}$ that are divisible by
$p$ for any prime $p$ which appears in the prime factorization of $m$.

In the following, we define natural numbers $c_j(i)$ with
$i=1,2,\ldots,m$ and $j=1,2,\ldots,k$ by
$$
\vartheta_{p_j}(m)={\displaystyle\sum^{m}_{i=1}c_j(i)}
$$
such that the $c_j(i)$'s are functions of ${\rm ord}_{p_j}(m-i+1)$
namely $c_j(i)=({\rm ord}_{p_j}(n-i+1))$ and $i=1,2,\ldots,m$, $j=1,2,\ldots,k$. Also $
c_j(i)=0\,\,{\rm if}\,\,{\rm ord}_{p_j}(m-i+1)=0$.

We now state and prove the following result.

\begin{theorem}\label{mxcj}
If $
\max\left\{c_j(i)\,\left|\,
\vartheta_{p_j}(m)={\displaystyle\sum^{m}_{i=1}c_j(i)}\right.
\right\}\leq\lfloor\log_{p_j}(m)\rfloor$ then $
\vartheta_{p_j}(m)\leq
\lfloor\frac{m}{p_j}\rfloor\lfloor\log_{p_j}(m)\rfloor$. (In general, the converse is not always true.) Therefore, a necessary but not sufficient condition in order to
satisfy the inequality
$\vartheta_{p_j}(m)\leq\lfloor\frac{m}{p_j}\rfloor\lfloor\log_{p_j}(m)\rfloor$,
is
$$
c_j(i)\leq\lfloor\log_{p_j}(m)\rfloor,\,\,\,
\forall\,i\in\llbracket 1,m\rrbracket
$$
with $j=1,2,\ldots,k$.
\end{theorem}

\begin{proof}

The proof is immediate from \eqref{eqa} by noticing the following,

$${\rm ord}_{p_j}(m!)=\sum_{l=1}^{\lfloor\log_{p_j}(m)\rfloor}\displaystyle\left\lfloor\frac{m}{p_j^l}\displaystyle\right\rfloor \leq \sum_{l=1}^{\lfloor\log_{p_j}(m)\rfloor}\displaystyle\left\lfloor\frac{m}{p_j}\displaystyle\right\rfloor =\displaystyle\left\lfloor\frac{m}{p_j^l}\displaystyle\right\rfloor \lfloor \log_{p_j}(m)\rfloor.$$

\end{proof}

We can notice that this choice is not unique. But, we can observe that all
the choices for the $c_j(i)$'s are equivalent in the sense that the equality
$\vartheta_{p_j}(m)=\sum^{m}_{i=1}c_j(i)$ should hold, meaning that we can come
back to a decomposition of the value of $\vartheta_{p_j}(m)$ into
sum of positive numbers like the $c_j(i)$'s for which
$c_j(i)\leq\lfloor\log_{p_j}(m)\rfloor$ with $i=1,2,\ldots,m$. It turns out that this choice is suitable in order to prove that
$\ell(m)$ is the minimal period of sequences $(a_n)$ such that
$a_n\equiv\binom{n}{m}\pmod m$ with
$m=\prod^k_{i=1}p^{b_i}_i$ (with at least one non-zero $b_i$).

\begin{remark}
We have obviously

$$
\max\left\{c_j(i)\,\left|\,
\vartheta_{p_j}(m)={\displaystyle\sum^{m}_{i=1}c_j(i)}\right.
\right\}\geq 1.
$$
\end{remark}

The above discussion gives us a motivation to study the coefficients $c_j(i)$'s. We hope to address a few issues related to them and establish some interesting
results in a forthcoming paper.

\bibliographystyle{amsplain}

\end{document}